\documentclass[a4paper]{amsart}
\usepackage{bbm}
\usepackage{graphicx}
\usepackage{amssymb}
\usepackage{amsmath}
\usepackage{amsthm}
\usepackage{amscd}
\usepackage[all,2cell]{xy}
\usepackage{mathbbold,bbm}
\usepackage{mathrsfs}
\usepackage{amsfonts}
\usepackage{xypic}
\usepackage{color}

\UseAllTwocells \SilentMatrices
\newtheorem{thm}{Theorem}[section]

\newtheorem{cor}[thm]{Corollary}
\newtheorem{lem}[thm]{Lemma}

\newtheorem{prop}[thm]{Proposition}
\newtheorem{prop-def}[thm]{Proposition-Definition}
\theoremstyle{definition}
\newtheorem{Def}[thm]{Definition}
\theoremstyle{remark}
\newtheorem{Rem}[thm]{\bf Remark}
\newtheorem{ex}[thm]{\bf Example}
\numberwithin{equation}{section}
\def\la{\lambda}
\def\op{{\rm op}}

\def\Q{\mathcal{Q}}
\def\U{\mathcal{U}}
\def\K{\mathbb{K}}
\def\ot{\otimes}
\def\deg{{\rm deg}}
\def\al{\alpha}
\def\arr{\overrightarrow}
\def\Ext{{\rm Ext}}
\def\xr{\xrightarrow}

\def\HQ{{\rm HQ}}
\def\H{{\rm H}}
\def\HH{{\rm HH}}
\def\HL{{\rm HL}}
\def\HC{{\rm HC}}
\def\LC{{\rm LC}}
\def\QC{{\rm QC}}
\def\Hom{{\rm Hom}}

\def\K{\mathbb{K}}
\def\vp{\varphi}
\def\id{{\rm id}}
\def\Im{{\rm Im}}
\def\Ker{{\rm Ker}}
\def\Der{{\rm Der}}
\def\1{\mathbbold{1}}

\begin{document}
\title[On quasi-Poisson cohomology]
{On quasi-Poisson cohomology}

\author[Y.-H. Bao and Y.Ye]
{Yan-Hong Bao\quad and Yu Ye}

\thanks{Supported by National Natural Science Foundation of China (No.10971206 and No.11126173).}
\subjclass[2010]{17B63, 20G10}
\date{\today}
\thanks{E-mail:
yhbao$\symbol{64}$ustc.edu.cn (Y.-H. Bao), yeyu@ustc.edu.cn (Y.Ye)}
\keywords{Poisson algebra, Quasi-Poisson module, Quasi-Poisson enveloping
algebra, Quasi-Poisson cohomology}

\maketitle

\dedicatory{}%
\commby{}%

\begin{abstract}
Let $A$ be a Poisson algebra and $\Q(A)$ its quasi-Poisson
enveloping algebra. In this paper, the Yoneda-Ext groups $\Ext^*_{\Q(A)}(A, A)$, which we call the quasi-Poisson
cohomology groups of $A$, are investigated. To calculate them we construct
a free resolution of $A$ as a $\Q(A)$-module. Moreover, we introduce the quasi-Poisson complex for $A$ to
simplify the calculation, and apply to obtain quasi-Poisson cohomology groups in some special cases. Finally,
by constructing a Grothendieck spectral sequence, we provide a way to calculate the quasi-Poisson cohomology
groups via the Hochschild cohomology and the Lie algebra cohomology. Examples are also shown here.
\end{abstract}

\section{Introduction}

Poisson algebra appears naturally in Hamiltonian mechanics. It plays
a central role in the study of Poisson geometry: A Poisson manifold is
a smooth manifold whose associative algebra of smooth functions is
equipped with the structure of a Poisson algebra. To meet the development
of the deformation of Poisson manifolds, it is helpful to introduce certain
deformation theory for Poisson algebras. For this one needs a more general
concept of Poisson structures. Note that there are many generalizations of
Poisson structures developed by different authors from different
perspectives, see for instance, \cite{Cr}, \cite{FGV}, \cite{RVO},
\cite{Van} and \cite{Xu}.

We are interested in the notion of Poisson algebra as introduced in \cite{FGV}.
By definition, a (non-commutative) \emph{Poisson algebra} over a field $\K$ consists
of a triple $(A,\cdot, \{-,-\})$, where $(A,\cdot)$ is an associative $\K$-algebra and
$(A,\{-,-\})$ is a Lie algebra over $\K$, such that the Leibniz rule $\{a, bc\}=\{a,
b\}c+b\{a, c\}$ holds for all $a,b,c\in A$. Unlike the classical ones, the associative
multiplication is not necessarily commutative. This is the most simple
way of generalization in algebraic nature, and has been studied by many authors
\cite{FL, Ku1, Ku2, Ku3, YYY, YYZ}.

We mention that this notion is different from
the non-commutative Poisson algebra developed by Xu \cite{Xu}.
As a natural generalization of commutative Poisson algebra, such a structure carries rich
information. Even in the study of commutative Poisson algebras, it will be quite useful.
For example, an interesting observation shows that the deformation quantization theory by
Kontsevich can be understood as a special kind of deformation within this version of Poisson
algebras, see \cite{BY} for more detail. In this paper and its sequel, we will make an attempt to investigate deeply
the cohomology theories for Poisson algebras.

A standard way to construct cohomology theory for an algebraic structure is to consider its module
category. Poisson module over a Poisson algebra was defined in \cite{FGV} in a natural way, see
also \cite{Ku2}. Later on in \cite{YYY}, the authors studied another version of module, say
quasi-Poisson module, over a Poisson algebra. Moreover, they introduced the (quasi-)Poisson
enveloping algebra for a Poisson algebra, and showed that the category of
(quasi-)Poisson modules is equivalent to the category of modules over the (quasi-)Poisson enveloping
algebra, see also section 2 below for detail. We emphasize that for commutative Poisson algebras, a
different version of Poisson enveloping
algebra has been introduced in \cite{oh}, see \cite[Remark 2.6]{YYY} for an explanation.

An immediate consequence is that the category of
(quasi-)Poisson modules has enough projective and injective objects,
which enables the construction of the cohomology theory for a Poisson
algebra by using projective or injective resolutions.
We aim to develop a cohomology theory, the quasi-Poisson cohomology, for Poisson algebras in this paper.
The ideal is quite easy. For a Poisson algebra $A$,
$A$ itself is a quasi-Poisson module, and naively we may study its Yoneda extension group
in the category of quasi-Poisson modules, which we call the quasi-Poisson cohomology group of $A$.

As we will show below, the quasi-Poisson cohomology relates to the Hochschild cohomology and Lie algebra
cohomology closely. In fact, there exists a Grothendieck spectral sequence, connecting the
 quasi-Poisson cohomology with the Hodchschild cohomology and the Lie algebra cohomology. In
 some extreme cases, the quasi-Poisson cohomology algebra is shown to be the tensor product
 of the Hochschild cohomology and the Lie algebra cohomology.

The quasi-Poisson cohomology plays an important role in the calculation of so-called Flato-Gerstenhaber-Voronov
Poisson cohomology as introduced in \cite{FGV}, a cohomology which controls the formal deformation of a
Poisson algebra, see \cite{BY} for details.
Another application is to the study of an associative algebra, say the quasi-Poisson cohomology group of
the standard Poisson algebra will give us an interesting invariant, carrying the
information of both the Hochschild cohomologies and the Lie algebra cohomologies.

Throughout $\K$ will be a field of characteristic zero. All algebras
considered are over $\K$ and we write $\ot=\ot_\K$ and $\Hom=\Hom_\K$ for brevity.
The paper is organized as follows.

Section 2 recalls basic definitions and notions.
In section 3, we introduce the concept of quasi-Poisson cohomology
groups for a Poisson algebra $A$, and in order to calculate it we construct a free resolution
of $A$ as a quasi-Poisson module, see Theorem \ref{proj.resol}. Moreover,
we introduce the quasi-Poisson complex for a Poisson algebra, which will simplify the
calculation. As an application,
some special quasi-Poisson cohomology groups are obtained in Section 4.
In Section 5, we construct a Grothendieck spectral sequence
for a smash product algebra, which connects the quasi-Poisson cohomology
with the Hochschild cohomology and the Lie algebra cohomology
when applying to a Poisson algebra.

\section{Preliminaries}

In this paper, we assume that all associative algebras will have a
multiplicative identity element.

Let $(A,\cdot,\{-,-\})$ be a Poisson algebra (not necessarily
commutative). A \emph{quasi-Poisson $A$-module} $M$ is a an
$A$-$A$-bimodule $M$ together with a $\K$-bilinear map
$\{-,-\}_\ast: A\times M\to M$, which satisfies
$$\{a,bm\}_\ast=\{a,b\}m+b\{a,m\}_\ast,\leqno (1.1)$$
$$\{a,mb\}_\ast=m\{a,b\}+\{a,m\}_\ast b,\leqno (1.2)$$
$$\{\{a,b\},m\}_\ast=\{a,\{b,m\}_\ast\}_\ast-\{b,\{a,m\}_\ast\}_\ast
\leqno (1.3)$$
for all $a,b\in A$ and $m\in M$. Clearly, the condition (1.3) just says
that $M$ is a Lie module over $A$. If moreover,
$$\{ab,m\}_\ast=a\{b,m\}_\ast+\{a,m\}_\ast b\leqno (1.4)$$
holds for all $a,b\in A$ and $m\in M$, then $M$ is called a \emph{Poisson $A$-module}. Let $M, N$ be quasi-Poisson modules  ({\it resp}. Poisson modules).
A homomorphism of quasi-Poisson $A$-module ({\it resp}. Poisson modules)
is a $\K$-linear function $f:M\to N$ which is
a homomorphism of both $A$-$A$-bimodules and Lie modules.

Let us recall the definition of (\emph{quasi}-)\emph{Poisson enveloping algebra} of a Poisson algebra,
see \cite{YYY} for more detail. Before that, we need some convention.

Denote by $A^\op$ the \emph{opposite algebra} of the associative
algebra $A$. Usually we use $a$ to denote an element in $A$ and $a'$ its counterpart in $A^\op$ to show the difference.
Let $\U(A)$ be the \emph{universal enveloping algebra} of the
Lie algebra $A$. Fix a $\K$-basis $\{v_i|i\in S\}$ of $A$,
where $S$ is an index set with a total ordering "$\le$". Let $\al=(i(1),\cdots,i(r))\in S^r$ be a sequence of length $r$ in $S$. Usually we call
$r$ the \emph{degree} of $\al$. Denote the element
$v_{i(1)}\ot \cdots\ot v_{i(r)}$ by $\arr{\al}$. If $i(1)\le \cdots \le i(r)$, then we call $\arr{\al}$ a
\emph{homogeneous element} of degree $r$. The empty sequence, or the sequence of degree 0, is denoted by $\varnothing$ and we write $\1=1_{\U(A)}=\arr{\varnothing}$ for brevity. Then
all homogeneous elements of positive degrees together with $\1$ form a
PBW-basis of $\U(A)$. For given $\al=(i(1),\cdots,i(r))$ and
$\beta=(j(1),\cdots,j(s))$, we define $\al\vee \beta:=(i(1),\cdots,
i(r),j(1),\cdots,j(s))$, and hence $\arr{\al}\arr{\beta}=\arr{\al\vee
\beta}$ under these notations.

Set $\underline{\rm{r}}=\{1,\cdots, r\}$. By an \emph{ordered bipartition} $\underline{\rm{r}} = X\sqcup Y$ of $\underline{\rm{r}}$, it is meant that $X$ and $Y$ are disjoint subsets of $\underline{\rm{r}}$ and $\underline{\rm{r}}= X\cup Y$, here
"ordered" means that $X\sqcup Y$ and $Y\sqcup X$ give different
bipartitions, which differs from the usual ones. Moreover, $X$ and $Y$ are allowed to be
empty sets. Let $\alpha=(i(1),\cdots,i(r))$ and $\underline{\rm{r}}= X\sqcup Y$.
Suppose $X=\{X_1,\cdots, X_{|X|}\}$ and $Y=\{Y_1,
\cdots, Y_{|Y|}\}$ with $X_1< X_2<\cdots<X_{|X|}$ and $Y_1< Y_2<\cdots< Y_{|Y|}$. Set
$\alpha_X=(i(X_1), \cdots, i(X_{|X|}))$ and $\alpha_Y=
(i(Y_1), \cdots, i(Y_{|Y|}))$. By definition
$\alpha=\alpha_{X}\sqcup\alpha_{Y}$ is called an
\emph{ordered bipartition} of $\alpha$ with respect to the ordered bipartition
$\underline{\rm{r}}=X\sqcup Y$.
 Similarly, one defines \emph{ordered $n$-partitions} $\alpha=
\alpha_1\sqcup\alpha_2\cdots\sqcup\alpha_n$ for any $n\ge2$.

The category of Lie modules over $A$ is
equivalent to the category of left $\U(A)$-modules. Notice that
$\U(A)$ is a cocommutative Hopf algebra with the "shuffle" coproduct
$\Delta(\arr{\al})=\sum\limits_{\al=\al_1\sqcup \al_2}
\arr{\al_1}\ot \arr{\al_2}$, where
the summation is taken over all possible ordered partitions of $\{1,\cdots, {\rm
deg}\,(\overrightarrow{\al})\}$, and the counit given by $\epsilon(\1)=1$,
$\epsilon(\arr{\al})=0$ for any $\arr{\al}$ of degree $\ge 1$. The Lie bracket makes $A$ a Lie
module and hence a $\U(A)$-module with the action given by
$\arr{\al}(a)=\{v_{i(1)},\{v_{i(2)},\{, \cdots,\{v_{i(r)},a\}\cdots
\}\}\}$ for $\al=(i(1),\cdots,i(r))\in S^r$ and $a\in A$. Moreover,
by the cocommutativity of $\U(A)$, the enveloping algebra $A^e=A\ot
A^\op$ of $A$ in the associative sense is a $\U(A)$-module algebra
with the action $\arr{\al}(a\ot b') =\sum\limits_{\al=\al_1\sqcup
\al_2} \arr{\al_1}(a)\ot \arr{\al_2}(b')$ for all $\al\in S^r$ with
$r\ge 0$ and $a\ot b'\in A^e$. Thus we have the following definition.

\begin{Def} {\rm (\cite{YYY})}
Let $A=(A,\cdot,\{-,-\})$ be a Poisson algebra. The smash product
$A^e\# \U(A)$ is called the \emph{quasi-Poisson enveloping algebra}
of $A$ and denoted by $\Q(A)$.
\end{Def}

\begin{Rem}{\rm By definition, $\Q(A)=A\ot A^\op\ot \U(A)$ as a $\K$-vector space.
Thus $\Q(A)$ has a PBW-basis given by
$$\{v_i\ot v'_j\# \arr{\al} \mid i,j\in S, \al=(i(1),\cdots,i(r))\in S^r,
i(1)\le \cdots \le i(r), r\ge 0\}.$$
The multiplication is given by
$$(v_{i_1}\ot v'_{j_1}\# \arr{\al})(v_{i_2}\ot v'_{j_2}\# \arr{\beta})
=\sum_{\al=\al_1\sqcup \al_2\sqcup \al_3}
(v_{i_1}\arr{\al_1}(v_{i_2}))\ot (v'_{j_1}(\arr{\al_2}(v_{j_2})')
\# (\arr{\al_3}\arr{\beta})$$
for $i_1,j_1,i_2,j_2\in S, \arr{\al}\in S^r,\arr{\beta}\in S^s, r,s\ge 0$,
and the identity in $\Q(A)$ is $1_A\ot 1_A'\# \1$.}
\end{Rem}

\begin{thm}{\rm (\cite{YYY})}
The category of quasi-Poisson modules
is equivalent to the category of $\Q(A)$-modules.
\end{thm}
To be precise, for a given a quasi-Poisson $A$-module $M$, one can define a $\Q(A)$-module $M$ with the action given by
\[(a\ot b' \#\arr{\al})m =a\arr{\al}(m)b\]
for all $m\in M$ and $a\ot b' \#\arr{\al}\in \Q(A)$. Conversely,
given a left $\Q(A)$-module $M$, we set
\[am=(a\ot 1'_A\# \1)m,~~ ma=(1_A\ot a'\#\1)m,~~\mbox{{\rm and}}~~ \{a,m\}_\ast=(1_A\ot 1'_A\# a)m\]
for all $m\in M, a\in A$ to obtain a quasi-Poisson module over $A$.

Consequently, there are enough projectives and injectives
in the category of quasi-Poisson modules,
which enables the construction of cohomology theories for a Poisson
algebra by using projective or injective resolutions in a standard way.

\section{Quasi-Poisson Cohomology}

One checks easily that under the action $\{-,-\}_\ast = \{-,-\}$, the regular
$A$-$A$-bimodule $A$ becomes a
quasi-Poisson module and hence a left module over $\Q(A)$. Then we
may consider the Yoneda-Ext groups $\Ext_{\Q(A)}^{\ast}(A,M)$ for any quasi-Poisson
module  $M$.

\begin{Def} Let $A$ be a Poisson algebra and $\Q(A)$ the quasi-Poisson
enveloping algebra of $A$. For any quasi-Poisson module $M$, the
extension group $\Ext^n_{\Q(A)}(A,M)$ is called the $n$-th
\emph{quasi-Poisson cohomology group} of $A$ with coefficient in the
quasi-Poisson module $M$, and denoted by $\HQ^n(A,M)$.
\end{Def}

\begin{Rem} The extension group $\HQ^n(A,A)$ is simply denoted by
$\HQ^n(A)$. One may consider the Yoneda-Ext algebra
$\HQ^\ast(A)=\bigoplus\limits_{n\ge 0} \HQ^n(A)$ with the multiplication
given by the Yoneda product, which is also called the \emph{quasi-Poisson cohomology algebra} of $A$.
Clearly, $\HQ^\ast (A)$ is
positively graded and each $\HQ^\ast (A,M)$ is a graded right $\HQ^\ast (A)$-module.
\end{Rem}

\subsection{A free resolution of $A$ as a $\Q(A)$-module}

In the sequel, we will construct a projective resolution of $A$ as a $\Q(A)$-module,
so that we can compute the cohomology groups $\Ext_{\Q(A)}^n(A, M)$ in a standard way.

To simplify notation, for each $i, j\ge 0$, we denote by $A^i$ and $\wedge^j$
the $i$-th tensor product and $j$-th exterior power of the $\K$-space $A$ respectively.

Our construction is based on the following two
well-known resolutions. One is
\[\mathbb{S}\colon\ \cdots\to A^{ i+2}\xr{\delta_i} A^{ i+1} \to \cdots \to A\ot A\ot A
\xr{\delta_1} A\ot A\xr{\delta_0} A\to 0,\]
the bar resolution of $A$ as an $A^e$-module($A$-$A$-bimodule), where
\[\delta_i(a_0\ot a_1\ot \cdots \ot a_{i+1})
=\sum_{k=0}^i (-1)^k a_0\ot \cdots \ot a_ka_{k+1}\ot \cdots \ot
a_{i+1}.\] The other one is the Koszul resolution of $\K$ as a trivial
$\U(A)$-module, say
\[\mathbb{C}: \cdots\to \U(A)\ot \wedge^j\xr{d_j} \U(A)\ot \wedge^{j-1}
\to \cdots \to \U(A)\ot \wedge^1\xr{d_1} \U(A)\xr{\epsilon} \K\to
0,\] where $\epsilon$ is the counit map, i.e., $\epsilon(\1)=1$, and
$\epsilon(\arr{\al})=0$ for all $r>0$ and $\al\in S^r$. The differential
is given by
\begin{equation*}
\begin{split}
d_j(\arr{\al}\ot v_1\wedge \cdots \wedge v_j) &= \sum_{l=1}^j
(-1)^{l+1} (\arr{\al}\ot v_l)
\ot (v_1\wedge \cdots\widehat{v}_l \cdots\wedge v_j)\\
&+\sum_{1\le p<q\le j}(-1)^{p+q}\arr{\al}\ot ( \{v_p,v_q\}\wedge
v_1\wedge\cdots\widehat{v}_p\cdots\widehat{v}_q \cdots\wedge
v_j),
\end{split}
\end{equation*}
where the symbol $\widehat{v}_l$ indicates that the term $v_l$ is to be
omitted.

Denote by  $\mathbb{S}'$ and $\mathbb{C}'$ the deleted complexes of $\mathbb{S}$ and $\mathbb{C}$ respectively.
Consider the double complex $\mathbb{S}'\ot\mathbb{C}'$,
{\footnotesize
\[\begin{CD}
&&  \cdots          &&         \cdots                &&     \cdots           &&\\
 &&  @VVV                  @VVV                    @VVV           &&\\
0 @<<< A^4\ot \U(A) @<<< A^4\ot \U(A)\ot \wedge^1 @<<< A^4\ot\U(A)\ot \wedge^2 @<<< \cdots\\
  &&  @VVV                  @VVV                    @VVV           &&\\
0 @<<< A^3\ot \U(A) @<<< A^3\ot \U(A)\ot \wedge^1 @<<< A^3\ot\U(A)\ot \wedge^2 @<<< \cdots\\
  &&  @VVV                  @VVV                    @VVV           &&\\
0 @<<< A^2\ot \U(A) @<<< A^2\ot \U(A)\ot \wedge^1 @<<< A^2\ot\U(A)\ot \wedge^2 @<<< \cdots\\
  &&  @VVV                  @VVV                    @VVV           &&\\
  &&  0          &&         0                &&     0           &&
\end{CD}\]
}
and obtain its total complex $\mathbb{Q}'= \mathrm{Tot}(\mathbb{S}'\ot\mathbb{C}')$,
\begin{align}\label{delete-quasi-res} \mathbb{Q}':
\cdots \to Q_n\xr{\varphi_n} Q_{n-1}\to \cdots \to Q_1\xr{\varphi_1}
Q_0\to 0.
\end{align}
 To be precise, $Q_0=A^2\ot \U(A)$, and for $n\ge 1$,
\begin{align*}
Q_n &=\bigoplus_{i+j=n}A^{i+2}\ot \U(A)\ot \wedge^j,\\
\vp_n &=\bigoplus_{i+j=n}(\delta_i\ot \id +(-1)^i\id\ot d_j).
\end{align*}

The following lemmas will be handy for later use. Some of them seem to be well known to experts.
For the convenience of the reader, we also include a proof.
\begin{lem} \label{proj}
For any $n\ge 0$,
$Q_n$ is a free module over $\Q(A)$ under
the action
\begin{align*}
&(v_{i_1}\ot v'_{j_1}\# \arr{\al})
(v_1\ot \cdots \ot v_i\ot \arr{\gamma}\ot \omega^j)\\
:=&\sum_{\al=\al_1\sqcup \cdots \sqcup \al_{i+1}}
v_{i_1}\arr{\al}_1(v_1)\ot\arr{\al}_2(v_2)\ot \cdots \ot
\arr{\al}_{i-1}(v_{i-1})\ot \arr{\al}_i(v_i)v_{j_1}\ot \arr{\al}_{i+1}\arr{\gamma}
\ot \omega^j
\end{align*}
for all $v_{i_1}\ot v'_{j_1}\# \arr{\al}\in \Q(A)$, $v_1\ot \cdots \ot v_i\ot \arr{\gamma}\ot \omega^j\in Q_n$.
\end{lem}
\begin{proof}[\textbf{Proof}]Firstly, we show that $Q_n$ is a left
$\Q(A)$-module. It suffices to check that the equality
\begin{equation*}
\begin{split}
&(a\ot b'\# \arr{\al})((c\ot d'\# \arr{\beta})
(v_{k(1)}\ot \cdots \ot v_{k(i)}\ot \arr{\gamma}\ot \omega^j))\\
=&((a\ot b'\# \arr{\al})(c\ot d'\# \arr{\beta})) (v_{k(1)}\ot \cdots
\ot v_{k(i)}\ot \arr{\gamma}\ot \omega^j)
\end{split}
\end{equation*} holds.
In fact,
\begin{equation*}
\begin{split}
LHS=&(a\ot b'\# \arr{\al})
((c\ot d'\# \arr\beta)(v_{k(1)}\ot \cdots \ot v_{k(i)}\ot \arr{\gamma}\ot \omega^j))\\
=&(a\ot b'\# \arr{\al}) \sum_{\beta=\beta_1\sqcup \cdots \sqcup
\beta_{i+1}} c\arr{\beta}_1(v_{k(1)}) \ot \cdots \ot
\arr{\beta}_i(v_{k(i)})d' \ot \arr{\beta}_{i+1}\arr{\gamma}
\ot \omega^j\\
=&\sum_{{\al=\al_1\sqcup \cdots\sqcup \al_{i+1}}\atop
{\beta=\beta_1\sqcup \cdots \sqcup \beta_{i+1}}}
a\arr{\al}_1(c\arr{\beta}_1(v_{k(1)})) \ot
\arr{\al}_2(\arr{\beta}_2(v_{k(2)}))
\ot \cdots\\
&\ot\arr{\al}_{k(i-1)}(\arr{\beta}_{i-1}(v_{k(i-1)})) \ot
\arr{\al}_i(\arr{\beta}_i(v_{k(i)})d')b' \ot
\arr{\al}_{i+1}\arr{\beta}_{i+1}\arr{\gamma} \ot \omega^j.
\end{split}
\end{equation*}
By the Leibniz rule $\{a,bc\}=\{a,b\}c+b\{a,c\}$, we have
$$\arr{\al}_1(c\arr{\beta}_1(v_{k(1)}))
=\sum_{\al=\xi_1\sqcup \xi_2} \arr{\xi}_1(c)
((\arr{\xi}_2\arr{\beta})(v_{k(1)})),$$
$$\arr{\al}_i(\arr{\beta}_i(v_{k(i)})d)
=\sum_{\al_i=\zeta_1\sqcup \zeta_2}((\arr{\zeta}_2\arr{\beta}_i)(v_{k(i)}))
\arr{\zeta}_1(d).$$
Hence,
\begin{equation*}
\begin{split}
LHS =& \sum_{{\al=(\xi_1\sqcup \xi_2)\sqcup \al_2\sqcup \cdots
\sqcup \al_{i-1} \sqcup(\zeta_1\sqcup \zeta_2)\sqcup \al_{i+1}}
\atop{\beta=\beta_1\sqcup \cdots \sqcup \beta_{i+1}}}
a\arr{\xi}_1(c)((\arr{\xi}_2\arr{\beta})(v_{k(1)})) \ot
\arr{\al}_2(\arr{\beta}_2(v_{k(2)}))
\ot \cdots\\
&\ot \arr{\al}_{i-1}(\arr{\beta}_{i-1}(v_{k(i-1)})) \ot
((\arr{\zeta}_2\arr{\beta}_i)(v_{k(i)})) \arr{\zeta}_1(d)b \ot
\arr{\al}_{i+1}\arr{\beta}_{i+1}\arr{\gamma} \ot \omega^j;
\end{split}
\end{equation*}
\begin{equation*}
\begin{split}
RHS=&((a\ot b'\# \arr{\al})((c\ot d'\# \beta))
(v_{k(1)}\ot \cdots \ot v_{k(i)}\ot \arr{\gamma}\ot \omega^j)\\
=&(\sum_{\al=\al_1\sqcup\al_2\sqcup \al_3}a \arr{\al_1}(c)\ot
(\arr{\al}_2(d)b)' \#\arr{\al}_3\arr{\beta}))(v_{k(1)}\ot \cdots \ot
v_{k(i)}\ot
\arr{\gamma}\ot \omega^j)\\
=&\sum_{{\al=\al_1\sqcup\al_2\sqcup \al_3}\atop
{\al_3\beta=\beta_1\sqcup\cdots\sqcup \beta_{i+1}}}
a\arr{\al_1}(c)\arr{\beta}_1(v_{k(1)}) \ot \arr{\beta}_2(v_{k(2)})
\ot \cdots \\
&\ot \arr{\beta}_{i-1}(v_{k(i-1)})\ot \arr{\beta}_i(v_{k(i)})
\arr{\al}_2(d)b \ot \arr{\beta}_{i+1}\arr{\gamma}\ot \omega^j.
\end{split}
\end{equation*}
From the significance of the notation $\sqcup$, we get
\[\sum_{\al_3\beta=\beta_1\sqcup\cdots\sqcup \beta_{i+1}}
=\sum_{{\al_3=\xi_1\sqcup \cdots \sqcup \xi_{i+1}}\atop
{\beta=\beta_1 \sqcup \cdots\sqcup \beta_{i+1}}}.\] Then we have
\begin{align*}
RHS =& \sum_{{\al=\al_1\sqcup \al_2\sqcup \xi_1\sqcup \cdots \sqcup
\xi_{i+1}}\atop {\beta=\beta_1 \sqcup \cdots\sqcup \beta_{i+1}}}
a\arr{\al_1}(c)(\arr{\xi}_1 \arr{\beta}_1(v_{k(1)})) \ot
\arr{\xi}_2\arr{\beta}_2(v_{k(2)})
\ot \cdots\\
&\ot \arr{\xi}_{i-1}\arr{\beta}_{i-1}(v_{k(i-1)}) \ot
\arr{\xi_i}\arr{\beta}_i(v_{k(i)})\arr{\al}_2(d)b
\#\arr{\xi_{i+1}}\arr{\beta}_{i+1}\arr{\gamma}\ot \omega^j).
\end{align*}
Comparing $LHS$ with $RHS$, we obtain the equality needed.

Next, we show that $Q_n$ is free
over $\Q(A)$ for each $n$. Set $Q_{ij}=A^{i+2}\ot \U(A)\ot \wedge^j$. We claim
that $Q_{ij}$ is a free $\Q(A)$-module with a basis \[\left\{1_A\ot v_{k(1)}\ot \cdots \ot v_{k(i)} \ot 1_A \ot \1
\ot v_{l(1)}\wedge \cdots \wedge v_{l(j)} \bigg | {\scriptstyle k(1),\cdots,k(i), l(1),\cdots, l(j)\in S \atop{\scriptstyle l(1)<\cdots <l(j), i,j\ge 0}}\right\}.\]

Notice that there exists a PBW-like basis of $\Q(A)$ given by
$v_{s}\ot v'_{t} \# \arr{\al}$, where $s,t\in S$ and $\arr{\al}$ is
a homogeneous element of degree $l$ in $\U(A)$. Following the
notations in \cite{YYY}, we write $\arr{\theta}=v_{k(1)}\ot \cdots
\ot v_{k(i)}$ if $\theta=(k(1),\cdots,k(i))\in S^i$, and
$\widetilde{\omega}=v_{l(1)}\wedge \cdots \wedge v_{l(j)}$ if
$\omega=(l(1),\cdots,l(j))$ with $l(1)<\cdots <l(j)$.

Assume that some $\Q(A)$-linear combination equals to zero, that is,
\[\sum\lambda_{s,t,\alpha,\theta,\omega}(v_{s}\ot v'_{t}\# \arr{\al})
(1_A\ot \arr{\theta}\ot 1_A\ot \1 \ot \widetilde{\omega})=0,\]
where each $v_{s}\ot v'_{t}\# \arr{\al}$ is chosen to be in the PBW-basis.
Let $\al$ be with highest degree which appears in the sum. Moreover, each term
in the left hand side is written as
\[(v_{s}\ot v'_{t}\# \arr{\al})
(1_A\ot \arr{\theta}\ot 1_A\ot \1 \ot \widetilde{\omega})
=\sum_{\al=\al_1\sqcup \al_2} v_{s} \ot \arr{\al_1}
(\arr{\theta})\ot v'_{t}\ot \arr{\al_2} \ot \widetilde{\omega},\]
where $\arr{\al_1}(\arr{\theta})
=\sum_{\al_1=\beta_1\sqcup \cdots \sqcup \beta_i}
\arr{\beta}_1(v_{k(1)})\ot \cdots \ot \arr{\beta}_i(v_{k(i)})$.

Combining those terms containing $\al$ in the resulting sum, we have
$$\sum\lambda_{s,t,\alpha,\theta,\omega} (v_{s}\ot \arr{\theta} \ot v'_{t} \ot \arr{\al}\ot \widetilde{\omega})=0.$$
Thus $\lambda_{s,t,\alpha,\theta,\omega}=0$ for any $s, t, \theta$ and
$\omega$, and our
claim follows.
\end{proof}

\begin{Rem}The corresponding quasi-Poisson action of $A$ on $A^{i}\ot
\U(A)\ot \wedge^j$ is given by
\begin{equation*}
\begin{split}
&\{a, v_1\ot \cdots \ot v_i\ot \arr{\beta}\ot
\omega^j\}_\ast\\
:=&\sum_{k=1}^i v_1\ot \cdots \ot \{a,v_k\} \ot \cdots\ot v_i
\ot \arr{\beta} \ot  \omega^j+v_1\ot \cdots \ot v_i \ot (\arr{a}\cdot \arr{\beta})\ot \omega^j
\end{split}
\end{equation*}
for all $a\in A$, $v_1\ot \cdots \ot v_i\in A^{i},  \arr{\beta}\in \U(A)$ and $
\omega^j\in \wedge^j$.
\end{Rem}

\begin{lem}\label{morph}
The morphisms $\varphi_n$ $(n=0,1,2,\cdots)$ in the total complex \eqref{delete-quasi-res} are
$\Q(A)$-homomorphisms.
\end{lem}
\begin{proof}[\textbf{Proof}] Clearly, each $\vp_n$ in
$\mathbb{Q}'$
is a direct sum of $\begin{pmatrix} \delta_i\ot \id \\
(-1)^i\id\ot d_j \end{pmatrix}$ by definition. It suffices
to show that $\delta_i\ot \id$ and $\id\ot d_j$ are homomorphisms of
$\Q(A)$-modules. Firstly, $\delta_i\ot \id$ and $\id\ot d_j$ are
$A^e$-homomorphisms and hence
$$(\delta_i\ot \id)((a\ot b'\# \1)x)=(a\ot b'\# \1)(\delta_i\ot \id)(x),$$
$$(\id\ot d_j)((a\ot b'\# \1)x)=(a\ot b'\# \1)(\id\ot d_j)(x),$$
for all $x\in A^{i+2}\ot\U(A)\ot \wedge^j$.

On the other hand,
for any $1_A\ot 1'_A\#\arr{\al}\in \Q(A)$,
{\footnotesize
\begin{equation*}
\begin{split}
&(\delta_i\ot \id)((1_A\ot 1'_A\#\arr{\al})(v_0\ot \cdots \ot v_{i+1} \ot \arr{\beta} \ot \omega^j))\\
=& (\delta_i\ot \id)(\sum_{\al=\al_0\sqcup \cdots \sqcup \al_{i+2}}
\arr{\al}_0(v_0)\ot \cdots \ot \arr{\al}_{i+1}(v_{i+1})\ot \arr{\al}_{i+2}\cdot \arr{\beta} \ot \omega^j)\\
=&\sum_{{\al=\al_0\sqcup \cdots \sqcup \al_{i+2}}\atop{0\le k\le
i}}\!\!\!\!\!\! (-1)^k \arr{\al}_0(v_0)\ot \cdots\ot
\arr{\al}_k(v_k)\arr{\al}_{k+1}(v_{k+1}) \ot\cdots
\ot \arr{\al}_{i+1}(v_{i+1})\ot \arr{\al}_{i+2}\cdot \arr{\beta} \ot \omega^j\\
=&\sum_{{\al=\al_0\sqcup \cdots \sqcup \al_{i+1}} \atop {0\le k\le i}}\!\!\!\!\!\!  (-1)^k
\arr{\al}_0(v_0)\ot \cdots\ot \arr{\al}_k(v_k v_{k+1}) \ot \arr{\al}_{k+1}(v_{k+2}) \ot\cdots
\ot \arr{\al}_{i}(v_{i+1})\ot \arr{\al}_{i+1}\cdot \arr{\beta} \ot \omega^j\\
=& (1_A\ot 1'_A\#\arr{\al})(\sum_{0\le k\le i}(-1)^k v_0\ot \cdots
\ot v_kv_{k+1}\ot \cdots \ot v_{i+1}
\ot \arr{\beta}\ot \omega^j)\\
=&(1_A\ot 1'_A\#\arr{\al})((\delta_i\ot \id)(v_0\ot \cdots \ot
v_{i+1} \ot \arr{\beta} \ot \omega^j))
\end{split}
\end{equation*}}
By the definition of $d_j$, it is easy to check that
\begin{equation*}
\begin{split}
&(\id\ot d_j)((1_A\ot 1'_A\#\arr{\al})(v_0\ot \cdots \ot v_{i+1} \ot \arr{\beta} \ot \omega^j))\\
=& (1_A\ot 1'_A\#\arr{\al})((\id\ot d_j)(v_0\ot \cdots \ot v_{i+1}
\ot \arr{\beta} \ot \omega^j)).
\end{split}
\end{equation*}
Since $a\ot b'\#\1, 1_A\ot 1'_A\# \arr{\al}$  generate $\Q(A)$, it follows that $\delta_i\ot \id$ and $\id\ot d_j$ are
$\Q(A)$-homomorphisms.
\end{proof}

\begin{lem} \label{exactness}
Keeping the above notations, we have
$$\H_0(\mathbb{Q}')\cong A, ~~\mbox{\rm and}~~\H_n(\mathbb{Q}')=0, n\ge 1.$$
\end{lem}
\begin{proof} [\textbf{Proof}] By K\"unneth's theorem (see \cite{HS}), it is easy to see that $\mathbb{Q}'$ is
exact at $Q_n$ for each $n\ge 1$, since both $\mathbb{S}'$ and $\mathbb{C}'$ are exact for $i,j>0$,
and $\mathbb{Q}'$ is the total complex of $\mathbb{S}'\ot \mathbb{C}'$.

For $n=0$, again by K\"unneth's theorem,
\[H_0(\mathbb{Q}')
\cong H_0(\mathbb{S}')\ot H_0(\mathbb{C}')=A\ot \K\cong A.\]
\end{proof}

Combining Lemma \ref{proj}, Lemma \ref{morph} and Lemma \ref{exactness}, we
obtain a projective resolution of $A$ as a $\Q(A)$-module.

\begin{thm} \label{proj.resol}
Let $A$ be a Poisson algebra, $\Q(A)$ the quasi-Poisson
enveloping algebra of $A$, and $\varphi_0\colon Q_0 \to A $ the $\Q(A)$-homomorphism given by
$\varphi_0(a_0\ot a_1\ot \arr{\al})=\epsilon(\arr{\al})a_0a_1$.
Then the sequence $\mathbb{Q}'$ together with the map $\varphi_0$, say
 \begin{equation}\label{quasi-res}
    \mathbb{Q}\colon\
\cdots \to Q_n\xr{\varphi_n} Q_{n-1}\to \cdots \to
Q_1\xr{\varphi_1} Q_0\xr{\varphi_0} A\to 0,
\end{equation}
is a free resolution of $A$ as a $\Q(A)$-module.
\end{thm}

Let $M$ be a left $\Q(A)$-module and hence a quasi-Poisson module
over $A$. Applying the functor $\Hom_{\Q(A)}(-,M)$ to the deleted
complex $\mathbb{Q}'$, we
obtain a complex $\Hom_{\Q(A)}(\mathbb{Q}', M)$:
\begin{equation*}
\begin{split}
0\to&  \Hom_{\Q(A)}(Q_0,M)\to \Hom_{\Q(A)}(Q_1,M)\to
\Hom_{\Q(A)}(Q_2,M)
\to \cdots\\
\to& \Hom_{\Q(A)}(Q_n, M)\to \Hom_{\Q(A)}(Q_{n+1}, M)\to \cdots.
\end{split}
\end{equation*}
By Theorem \ref{proj.resol}, the $n$-th quasi-Poisson cohomology group is calculated by
$$\HQ^n(A,M)=\Ext_{\Q(A)}^n(A,M)=H^n \Hom_{\Q(A)}(\mathbb{Q}', M).$$

\subsection{Quasi-Poisson complex}

To compute the quasi-Poisson cohomology groups, one uses a simplified complex, the so-called
\emph{quasi-Poisson complex}.
Let $M$ be a quasi-Poisson module.
Applying the functor $\Hom_{\Q(A)}(-, M)$ to the bicomplex $\mathbb{S}'\ot \mathbb{C}'$, we obtain
{\footnotesize
\[\begin{CD}
&&  \cdots          &&         \cdots                &&    \\
 &&  @AAA                  @AAA                   \\
0 @>>> \Hom_{\Q(A)}(A^4\ot \U(A),M) @>>> \Hom_{\Q(A)}(A^4\ot \U(A)\ot \wedge^1,M) @>>>  \cdots\\
 &&  @AAA                  @AAA                   \\
0 @>>> \Hom_{\Q(A)}(A^3\ot \U(A),M) @>>> \Hom_{\Q(A)}(A^3\ot \U(A)\ot \wedge^1,M) @>>>  \cdots\\
 &&  @AAA                  @AAA                   \\
0 @>>> \Hom_{\Q(A)}(A^2\ot \U(A),M) @>>> \Hom_{\Q(A)}(A^2\ot \U(A)\ot \wedge^1,M) @>>>  \cdots\\
&&  @AAA                  @AAA                   \\
  &&  0          &&         0
\end{CD}\]
}
Following from the natural $\K$-isomorphisms
$$\mathrm{\Phi}_{i,j}\colon\ \Hom_{\Q(A)}(A^{i+2}\ot \U(A)\ot \wedge^j, M)\xr{\simeq }\Hom(A^i\ot \wedge^j, M),$$
$$\mathrm{\Phi}_{i,j}(f)((a_1\ot \cdots \ot a_i) \ot (x_1\wedge \cdots \wedge x_j))= f(1_A \ot (a_1\ot \cdots \ot a_i)\ot 1_A \ot \1 \ot (x_1\wedge \cdots \wedge x_j)),$$
 the above bicomplex is isomorphic to the bicomplex $\Hom(A^\bullet \ot \wedge^\bullet, M)$:
{\footnotesize
\[\begin{CD}
&& \cdots  && \cdots   && \cdots \\
&&@AAA @AAA @AAA\\
0 @>>> \Hom(A^2,M) @>{\sigma_H^{1,0}}>> \Hom(A^2\ot \wedge^1, M) @>{\sigma_H^{1,1}}>> \Hom(A^2\ot \wedge^2,M)@>>> \cdots\\
&&@A{\sigma_V^{1,0}}AA @A{\sigma_V^{1,1}}AA @A{\sigma_V^{1,2}}AA\\
0 @>>> \Hom(A,M) @>{\sigma_H^{1,0}}>> \Hom(A\ot \wedge^1, M) @>{\sigma_H^{1,1}}>> \Hom(A\ot \wedge^2,M)@>>> \cdots\\
&&@A{\sigma_V^{0,0}}AA @A{\sigma_V^{0,1}}AA @A{\sigma_V^{0,2}}AA\\
0 @>>> M @>{\sigma_H^{0,0}}>> \Hom(\wedge^1, M) @>{\sigma_H^{0,1}}>> \Hom(\wedge^2,M)@>>> \cdots\\
&&@AAA @AAA @AAA\\
&& 0 && 0&& 0
\end{CD}\]
}
where
\begin{align*}
&(\sigma_V^{i,j}(f))((a_1\ot \cdots \ot a_{i+1}) \ot (x_1\wedge\cdots \wedge x_j))\\
= &\ a_1 f ((a_2\ot \cdots \ot a_i) \ot (x_1\wedge \cdots \wedge x_j))\\
&+\sum_{k=1}^{i}(-1)^k f ((a_1\ot \cdots \ot a_ka_{k+1}\ot \cdots \ot a_{i+1}) \ot (x_1\wedge \cdots \wedge x_j))\\
&+(-1)^{i+1}f((a_1\ot \cdots \ot a_{i})\ot (x_1\wedge \cdots \wedge x_j))a_{i+1},
\end{align*}
\begin{align*}
&(\sigma_H^{i,j}(f))((a_1\ot \cdots \ot a_{i}) \ot (x_1\wedge\cdots \wedge x_{j+1}))\\
= & \sum_{l=1}^{j+1}(-1)^{l+1}
 \bigg[
 \{x_l, f((a_1\ot \cdots \ot a_i) \ot(x_1\wedge \cdots \widehat{x}_l\cdots \wedge x_{j+1}))\}_\ast
\\
&~~~~~~~~~~~~~~~~~~~
-\sum_{t=1}^i f ((a_1\ot \cdots \ot \{x_l,a_t\}\ot\cdots\ot a_i) \ot (x_1\wedge \cdots  \widehat{x}_l\cdots \wedge x_{j+1}))
\bigg]\\
&+\sum_{1\le p<q\le {j+1}}(-1)^{p+q}f ((a_1\ot \cdots \ot a_i) \ot
(\{x_p,x_q\}\wedge x_1\wedge\cdots\widehat{x}_p\cdots \widehat{x}_q
\cdots \wedge x_{j+1}))
\end{align*}
for all $f \in \Hom(A^{i}\ot \wedge^j,A)$, and $(a_1\ot \cdots a_i)\ot (x_1\wedge \cdots\wedge x_j) \in A^i\ot \wedge^j$, $i,j\ge 0$.

\begin{Rem}\label{rem-hoch-kos-cpx}
Write $\delta^n=\sigma_V^{n,0}$ and $d^n=\sigma_H^{0,n}$ for each $n\ge 0$. Clearly,
the leftmost row $(\Hom(A^\bullet, M), \delta^n)$ is nothing but the Hochschild complex $\HC(A, M)$
(see \cite{Ha1,Ho}),
and the bottom column $ \LC(A, M) = (\Hom(\wedge^\bullet, M), d^n)$ calculates the Lie algebra cohomology $\Ext_{\U(A)}^*(\K, M)$.

Let $\HH^n(A, M)$ denote the $n$-th Hochshcild cohomology of $A$ with coefficients in the $A$-bimodule $M$, and $\HL^n(A, M)= \Ext_{\U(A)}^n(\K, M)$
the $n$-th Lie algebra cohomology of the Lie algebra $A$ with coefficients in the Lie module $M$. Thus $\HH^n(A, M)= H^n(\HC(A, M))$ and $\HL^n(A,M)=H^n(\LC(A, M))$.
\end{Rem}

\begin{Def}\label{quasi-Poisson complex}
Let $A$ be a Poisson algebra and $M$ a quasi-Poisson module. The total
complex of $\Hom(A^\bullet \ot \wedge^\bullet, M)$, say
\begin{align*}
0\to&  M\xr{\sigma^0} \Hom_{\K}(A\oplus \wedge^1, M)\xr{\sigma^1}
\Hom_{\K}(A^{2}\oplus A\ot \wedge^1 \oplus \wedge^2, M)
\xr{\sigma^2} \cdots \\
\to& \Hom_{\K} (\bigoplus_{i+j=n}A^{i}\ot \wedge^j, M)
\xr{\sigma^{n}} \Hom_{\K}(\bigoplus_{i+j=n+1}A^{i}\ot \wedge^j, M)\to \cdots,
\end{align*}
\[\sigma^n =\bigoplus\limits_{i+j=n} \left(\sigma_V^{i,j} +(-1)^i  \sigma_H^{i,j}\right)\ \forall\ n\ge0,\]
is called the \emph{quasi-Poisson complex} of $A$ with
coefficients in $M$, and denoted by $\mathrm{QC}(A, M)$.
\end{Def}

An immediate consequence follows.
\begin{prop} \label{equ-complex}
The quasi-Poisson complex  $\mathrm{QC}(A, M)$ is isomorphic to
the complex $\Hom_{\Q(A)}(\mathbb{Q}',M)$, and hence
$\H^n(\mathrm{QC}(A, M))=\HQ^n(A,M)$ for all $n\ge 0$.
\end{prop}

\subsection{Lower-dimensional quasi-Poisson cohomologies}

First examples are lower dimensional quasi-Poisson cohomology groups of a
Poisson algebra $(A,\cdot,\{-,-\})$. We denote by $Z(A)$ and $Z\{A\}$ the center of the associative algebra
and the one of the Lie algebra, respectively. Then we have the following easy result.

\begin{prop}\label{hq0} Keep the above notation. Then $\HQ^0(A)=Z(A)\cap Z\{A\}$.
\end{prop}

Denote by $\Der(A)$ and $\Der_L(A)$ the $\K$-space of associative derivations and the space of Lie derivations respectively.
Consider the maps $\mathrm{ad}\colon A\to \Der(A)$ and $\mathrm{ad}_L\colon A\to \Der_L(A)$ given by $\mathrm{ad}(a)= [-,a]$
and $\mathrm{ad}_L(a)=\{-,a\}$ for all $a\in A$. Under these notations, the differential $\sigma^0=(\mathrm{ad}, \mathrm{ad}_L)$.

Moreover, for any $f=(f_1,f_0)\in \Ker\sigma^1$, by Proposition \ref{equ-complex}, we
know that $f_1\in\Der(A)$ and $f_0\in{\Der_L(A)}$ and the equality
\begin{equation}\label{1-qp-cohomology}f_1(\{x,a\})-\{x,f_1(a)\}_\ast=f_0(x)a-af_0(x)\end{equation}
holds for any $(a,x)\in A\oplus \wedge^1$. Now set
 \[D(A)=\{(f_1,f_0)\in \Der(A)\oplus \Der_L(A)\mid (\ref{1-qp-cohomology}) \text{ holds for all } a, x \in A\}.\]
Thus $\HQ^1(A) $ is computed as follows by definition.

\begin{prop} Keep the above notations.
 Then we have $\HQ^1=D(A)/\Im \sigma^0$ and hence \[\dim_{\K}\HQ^1(A) = \dim_{\K}D(A) - \dim_{\K} A + \dim_\K\HQ^0(A).\]
\end{prop}

\section{Examples}
\subsection{Standard Poisson algebras}

Let $A$ be an associative algebra. For any $a,b\in A$, we denote by
$[a,b]$ the commutator $ab-ba$ of $a$ and $b$.
Then $(A,\cdot, \la[-,-])$ is a Poisson algebra for a fixed $\la\in \K$, and we call it a \emph{standard Poisson algebra}.
By Proposition \ref{hq0}, we have $\HQ^0(A)=Z(A)$.

More generally, $\HQ^0(A)=Z(A)$ for any inner Poisson algebra since $Z(A)\subset Z\{A\}$ in this case,
see Lemma 1.1 in \cite{YYZ} for more details. Recall that a Poisson algebra $(A,\cdot,\{-,-\})$
is said to be inner if the Hamilton operator ${\rm ham}(a):=\{a,-\}$ is an inner derivation of $(A,\cdot)$ for any $a\in A$.

Now we turn to $\HQ^1$. Given $f_1\in\Der(A)$ and $f_0\in{\Der_L(A)}$. Note that in standard case, the equality $(\ref{1-qp-cohomology})$ is equivalent to
\[\Im(f_0-f_1)\in Z(A),\]
which holds true if and only if $f_1=f_0+g$ for some  Lie derivation $g$ satisfying $\Im g \subseteq Z(A)$.
Since $g([x,y])=[g(x),y]+[x,g(y)]$, we have $\Ker(g)\supseteq [A,A]$, thus $g$ is
obtained from some $\widetilde{g}\in\Hom(A/[A,A], Z(A))$. Conversely, each $\widetilde{g}\in\Hom(A/[A,A], Z(A))$
gives to a Lie derivation $g$ with $\Im g \subseteq Z(A)$. Thus we have the following characterization.

\begin{lem} Let $A$ be a standard Poisson algebra. Then \[\HQ^1(A)\cong \HH^1(A)\oplus \Hom(A/[A,A], Z(A)),\]
where $\HH^*(A)=\HH^*(A,A)$ is the Hochschild cohomology with coefficients of $A$ in itself.
\end{lem}

In general, quasi-Poisson cohomology groups of higher degrees are difficult to compute,
only some special cases are known to us.

\begin{ex} Let $A$ be the $\K$-algebra of upper triangular $2\times 2$ matrices. It is known to be the
path algebra of the quiver of $\mathbb{A}_2$ type. More explicitly, $A$ has a basis $\{e_1, e_2, \alpha\}$, and the multiplication is given by $e_ie_j=\delta_{ij}e_i$, $\alpha e_1=e_2\alpha=0$  and $e_1\alpha=\alpha e_2=\alpha$,
where $\delta_{ij}$ is the Kronecker sign. Clearly $1_A=e_1+e_2$.

Consider the standard Poisson algebra. By direct computation, one shows that  as a graded algebra, $\HQ^*(A)\cong \K\langle x,y\rangle/\langle x^2, y^2, xy+yx\rangle$, the exterior algebra in 2 variables. The grading is given by $\deg(x)=\deg(y)=1$.
\end{ex}

\begin{ex}  Consider the standard Poisson algebra of  $A=\mathbb{M}_2(\K)$, the $\K$-algebra of $2\times 2$ matrices. Again direct calculation shows that $$\HQ^0(A)=\HQ^1(A)=\HQ^3(A)=\HQ^4(A)=\K,$$ and $\HQ^i=0$ for $i\ne 0,1,3,4$. In fact, as a graded algebra, $$\HQ^*(A)\cong \K\langle x,y\rangle/\langle x^2, y^2, xy+yx\rangle,$$
where the grading is given by $\deg(x) = 1$ and $\deg(y)=3$.
\end{ex}

\subsection{Poisson algebras with trivial Lie bracket}

Let $(A,\cdot,\{-,-\})$ be a finite-dimensional Poisson algebra with trivial Lie
structure, i.e. $\{a,b\}=0$ for any $a,b\in A$. Clearly, $\Q(A)=A\otimes A^\op\otimes \U(A)$ and $\U(A)\cong \mathcal{S}(A)$, where
$\mathcal{S}(A)$ is the polynomial algebra of the vector space $A$.

One shows easily that as a $\Q(A)$-module, $A$ is the tensor product of the $A^e$-module $A$ and the trivial Lie module $\K$ over $A$.
Then by a classical result in homological algebra,  $\HQ^*(A)\cong \HH^*(A)\otimes \Ext^*_{\mathcal{S}(A)}(\K,\K)$; see for instance, Theorem 3.1 in \cite{CE}, Chapter XI. By Koszul duality,
$\Ext^*_{\mathcal{S}(A)}(\K,\K)\cong \wedge A$, the exterior algebra of the vector space $A$.
Thus we have the following result.

\begin{prop}\label{prop-trivial-bracket}Let $(A,\cdot,\{-,-\})$ be a finite-dimensional Poisson algebra with trivial Lie
bracket. Then
$\HQ^*(A)\cong \HH^*(A)\otimes \wedge A$.
\end{prop}

\subsection{Poisson algebras with finite Hochschild cohomology dimension}

Let $(A, \cdot, \{-,-\})$ be a Poisson algebra. Suppose the associative algebra $A$ has finite Hochschild cohomology dimension, that is, the $n$-th Hochschild cohomology group
of $(A,\cdot)$ vanishes for sufficiently large $n$.

\begin{prop}\label{prop-f-hhdim}
Let $(A, \cdot, \{-,-\})$ be a Poisson algebra and $k$ a fixed positive integer.
Suppose $\HH^n(A)=0$ for all $n> k$.
Set $\Omega_n^{k}= \Hom(\bigoplus\limits_{i+j =n, i\le k } A^i\otimes \wedge^j, A)$.
Then the $n$-th quasi-Poisson cohomology group
$$\HQ^n(A)=\frac{\Ker\sigma^n\cap \Omega_n^{k}}{\Im\sigma^{n-1}\cap \Omega_n^{k}}.$$
\end{prop}
\begin{proof}[\textbf{Proof}] To compute the quasi-Poisson cohomology, again we use the quasi-Poisson complex $\QC(A, A)$.
Consider the $\K$-linear map $\pi\colon \Ker \sigma^n \cap \Omega_n^k \to \HQ^n(A)$,
$f\mapsto f+\Im \sigma^{n-1}$.
Suppose $f=(f_n,\cdots,f_1,f_0)\in \Ker \sigma^n$ for some $n>k$.
By definition
$f_n$ is an $n$-th cocycle in the Hochschild complex,  and
hence there exists some $g_{n-1}\in \Hom(A^{n-1},A)$ such that
$f_n=\delta^{n-1}g_{n-1}$ since the $n$-th Hochschild cohomology group vanishes, where $\delta^n$ is the $\K$-linear map in
the Hochschild complex. Clearly,  $\overline{f}=\overline{f-\sigma^n g}\in \HQ^n(A)$. Thus,
$$f-\sigma^n g=(0,\widetilde{f_{n-1}}, f_{n-2},\cdots,f_0),$$
For brevity, we still denote $\widetilde{f_{n-1}}$ by $f_{n-1}$. Therefore
\begin{equation*}
\begin{split}
& a_1f_{n-1}(a_2\ot \cdots\ot a_{n}\ot x)\\
&+\sum_{k=1}^{n-1}f_{n-1}(a_1\ot \cdots \ot a_{k}a_{k+1} \ot \cdots \ot a_n\ot x)\\
&+(-1)^{n+1}f_{n-1}(a_1\ot \cdots \ot a_{n-1}\ot x)a_n=0
\end{split}
\end{equation*}
If $n-1>k$, consider the $\K$-linear isomorphism $\Hom(A^{n-1}\ot \wedge^1, A)\to \Hom(A^{n-1}, A)\ot A^\ast$
such that $f_{n-1}\mapsto f'_{n-1}\ot f''_1$, where $A^\ast$ is the dual $\K$-vector space of $\wedge^1=A$.
Clearly, we have $f'_{n-1}\in \Ker \delta^{n-1}$.
By the assumption $\HH^{n-1}(A)=0$, there exists
$g'_{n-2}\in \Hom(A^{n-2},A)$ such that
$f'_{n-1}=\delta^{n-2}(g'_{n-2})$.
Suppose $g_{n-2}=g'_{n-2}\ot f''_1\in \Hom(A^{n-2},A)\ot A^\ast\cong \Hom(A^{n-2}\ot \wedge^1,A)$ and
$g=(0, g_{n-2},0,\cdots,0)$, then we have $f_{n-1}=(\sigma^{n-1}(g))_{n-1}$ and
$\overline{f}=\overline{f-\sigma^{n-1}(g)}\in \HQ^{n-1}(A)$. Clearly,
$$f-\sigma^{n-1}(g)=(0,0,\widetilde{f_{n-2}},f_{n-3},\cdots,f_0).$$
Denote again $\widetilde{f_{n-2}}=f_{n-2}$.

Repeat the above argument, we know that each $f\in \HQ^n(A)$ can be written as
\[\overline{f}=\overline{(0,\cdots, 0, f_{k},\cdots,f_0)}.\]

Therefore, the $\K$-homomorphism $\pi$ is surjective. Clearly, $\Ker \pi=\Omega_n^k \cap \Im \sigma^{n-1}$, and hence
$$\HQ^n(A)=\frac{\Ker\sigma^n\cap \Omega_n^k}{\Im\sigma^{n-1}\cap \Omega_n^k}.$$
\end{proof}

\section{A Grothendieck spectral sequence for quasi-Poisson cohomology}
In this section, we  construct a Grothendieck spectral
sequence for smash product algebras, and apply it to the calculation of extensions of quasi-Poisson modules.
As a special case, this Grothendieck spectral sequence
exhibit a close relation among the quasi-Poisson cohomology,
the Hochschild cohomology and the Lie algebra cohomology.

We begin with a general situation. Let $H$ be a Hopf algebra over $\K$ with
the comultiplication $\Delta$ and the antipode $S$. let $A$ be a module algebra
over $H$ and $A\# H$ be the smash product.
If $M,N$ are modules over $A\# H$, then $\Hom_A(M,N)$ is an $H$-module
with the action given by $(hf)(x)=\sum h_2f(S^{-1}h_1x)$
for $x\in M$, where we use the Sweedler's sum notation, see \cite{Sw}.
It is easy to show the natural isomorphism $\Hom_H(\K,\Hom_A(M,N))\cong \Hom_{A\# H}(M,N)$.
Thus, we have the following well-known lemma
which is crucial in our calculation.

\begin{lem}
Keep the above notations. Then we have the natural isomorphism of bifunctors
\[\Hom_H(\K,\Hom_A(-,-))\cong \Hom_{A\# H}(-,-).\]
\end{lem}
\begin{proof}[\textbf{Proof}]
For any $A\# H$-modules $X, Y$, the natural isomorphism
$$\Hom_H(\K,\Hom_A(X,Y))\xrightarrow{\simeq} \Hom_{A\# H}(X,Y)$$ is given by
$(\1 \mapsto f)\mapsto f$. The only left is routine check and we omit here.
\end{proof}

Applying the Grothendieck spectral sequence \cite[Thoerem 10.47]{Ro},
we obtain a spectral sequence for a smash product algebra.
This spectral sequence should be well known to experts, although we could not
find any reference for it.

\begin{thm}
Keep the above notations.
Then we have a spectral sequence
\[\Ext_H^q(\K,\Ext_A^p(M,N))\Longrightarrow \Ext_{A\# H}^{p+q}(M,N).\]
\end{thm}

Consequently, we obtain a Grothendieck spectral sequence
which is handful in calculating Yondena extensions in quasi-Poisson cohomologies.

\begin{cor}
Let $\Q(A)$ the quasi-Poisson enveloping algebra of the Poisson algebra $A$ and
$M, N$ be modules over $\Q(A)$.
Then we have a spectral sequence
\[\Ext_{\U(A)}^q (\K, \Ext_{A^e}^p(M,N))\Longrightarrow \Ext_{\Q(A)}^{p+q}(M,N).\]
\end{cor}

In particular, if we take $M=A$, then we obtain a spectral sequence connecting the Hochschild cohomology
of the associative algebra $A$,
the Cartan-Eilenberg cohomololgy of the Lie algebra $A$ and the quasi-Poisson cohomology of the Poisson algebra $A$.

\begin{thm}\label{spec seq for quasi-Poisson coh}
 Let $A$ be a Poisson algebra and $N$ a quasi-Poisson $A$-module. Then we have
a spectral sequence
\[\HL^q (A, \HH^p(A,N))\Longrightarrow \HQ^{p+q}(A,N).\]
\end{thm}

\begin{cor}\label{HH=0 for n ge 2}
Let $(A, \cdot, \{-,-\})$ be a Poisson algebra with $\HH^p(A)=0$ for all $p>1$.
Then we have the short exact sequence
\[0\to \HL^{n-1}(A, \Der^{\rm o}(A)) \to \HQ^n(A)\to \HL^n(A, Z(A))\to 0\]
for $n\ge 1$, where $\Der^{\rm o}(A))$ is the space of outer derivations of $A$.
\end{cor}
\begin{proof}[\textbf{Proof}]
By the assumption, we have
\begin{center}
$\HH^p(A)=0$ $(\forall p\ge 2)$, $\HH^1(A)=\Der^{\rm o}(A)$ and $\HH^0(A)=Z(A)$.
\end{center}
From Proposition 2.4 in \cite{Fu}, we have the following short exact sequence
\[0\to \HL^{n-1}(A, \Der^{\rm o}(A)) \to \HQ^n(A)\to \HL^n(A, Z(A))\to 0\]
for $n\ge 1$.
\end{proof}

\begin{cor}\label{prop-trivial-hoch}
Let $(A, \cdot, \{-,-\})$ be a Poisson algebra over $\K$ with $\HH^i(A)=0$ for all $i>0$. Then
$\HQ^*(A)\cong Z(A)\otimes \HL^*(A,\K)$ as graded algebras.
\end{cor}
\begin{proof}[\textbf{Proof}]
By Corollary \ref{HH=0 for n ge 2}, we have
\[\HQ^n(A)\cong \HL^n(A, Z(A)).\]
Furthermore, $A$ is an inner Poisson algebra and hence $Z(A)\subset Z\{A\}$ since $A$ has only inner derivations.
So we have
\[\HQ^*(A)\cong Z(A)\otimes \HL^*(A,\K).\]
\end{proof}

\begin{cor}
Let $A$ be a finite dimensional Poisson algebra. If ${\rm gl.dim} A <\infty$ as an associative algebra,
then ${\rm proj.dim}_{\Q(A)}A <\infty$ and $\HQ^\ast(A)$ is finite-dimensional.
\end{cor}
\begin{proof}[\textbf{Proof}] Let $M$ be a quasi-Poisson module over $A$.
Since ${\rm proj.dim}_{A^e}A={\rm gl.dim} A<\infty$, we have
$\HH^p(A,M)=0$
for $p\gg 0$. On the other hand, $\HL^q(A, -)=0$ for $q\gg 0$ since $A$ is finite-dimensional.
By the spectral sequence in Theorem \ref{spec seq for quasi-Poisson coh}, we have
$\Ext_{\Q(A)}^n(A,M)=\HQ^n(A, M)=0$ for $n\gg 0$ and hence ${\rm proj.dim}_{\Q(A)}A <\infty$
and $\HQ^\ast(A)$ is finite-dimensional.
\end{proof}

\begin{ex}
Let $Q$ be a finite connected quiver with underlying graph being a tree. Denote by $\K Q$ the path algebra of $Q$.
Then we have $\HH^p(\K Q)=0$ for any $p\ge 1$, see Section 1.6 in  \cite{Ha1}. Consider the standard Poisson algebra of $\K Q$,
by Proposition \ref{prop-trivial-hoch} it is immediate that $\HQ^n(\K Q)=\HL^n(\K Q, \K)$, the usual $n$-th Lie algebra cohomology group of
$(A,[-,-])$ with coefficients in $\K$.
\end{ex}

\begin{ex}
Let $Q$ be the 2-Kronecker quiver and $A$ be the path algebra of $Q$.
Then we have $\HH^p(A)=0$ for all $p\ge 2$, $\HH^1(A)=\K^3$, and $\HH^1(A)=\K$.
It follows from some easy calculations that $\HH^1(A)$ is a trivial module over the Lie algebra
$(A, [-,-])$. Consider the standard Poisson algebra of $A$, by some direct by tedious calculations, we have $\HL^0(A, \K)=\K$, $\HL^1(A,\K)=\K^2$, $\HL^2(A,\K)=\K$,
and $\HL^p(A,\K)=0$ for any $p\ge 3$. By Corollary \ref{HH=0 for n ge 2}, we have $\HQ^0(A)=\K$,
$\HQ^1(A)=\K^5$, $\HQ^2(A)=\K^7$, $\HQ^3(A)=\K^3$ and $\HQ^n(A)=0$ for all $n\ge 4$.

\end{ex}

\vspace{1cm}

{\footnotesize
\noindent Yan-Hong Bao \\
School of Mathematical Sciences, Anhui University, Hefei 230039, China \\
School of Mathematical Sciences, University of Science and Technology of China,
Hefei 230036, China\\
E-mail address: yhbao@ustc.edu.cn

\vspace{5mm}
\noindent Yu Ye\\
School of Mathematical Sciences, University of Science and Technology of China,
Hefei 230036, China \\
E-mail address: yeyu@ustc.edu.cn
}

\end{document}